\documentclass[journal]{IEEEtran}
\usepackage{amsmath, amssymb, bbm, mathtools, amsthm}
\usepackage{graphicx}
\usepackage{subfigure}
\usepackage{verbatim}
\usepackage{soul}

\newtheorem{theorem}{Theorem}
\newtheorem{corollary}{Corollary}
\newtheorem{remark}{Remark}
\newtheorem{lemma}[theorem]{Lemma}

\usepackage{hhline}
\usepackage{multirow}

\usepackage{algorithm,algorithmic}

\usepackage{cite}

\newcommand{\conj}[1]{\overline{#1}}
\newcommand{\br}[1]{\left({#1}\right)}
\newcommand{\Com}{\mathbb{C}}

\newcommand{\xdiff}{\tilde{x}}
\newcommand{\limy}{\ell_x}
\newcommand{\limu}{\ell_u}

\newcommand{\limq}{\ell_q}

\newcommand{\udiff}{\tilde{u}}

\newcommand{\Acal}{\mathcal{A}}

\usepackage{cleveref}
\usepackage{tcolorbox}

\begin{document}

\title{Constructing Convex Inner Approximations of Steady-State Security Regions}
\author{Hung D. Nguyen, Krishnamurthy Dvijotham, and~Konstantin~Turitsyn

\thanks{HN is with School of EEE, Nanyang Technological University, Singapore email:hunghtd@ntu.edu.sg; KD is with Google Deepmind, London, UK. e-mail:dvij@cs.washington.edu; KT is with the Department of Mechanical Engineering, MIT, Cambridge, MA, USA e-mail: turitsyn@mit.edu}}

\markboth{IEEE Transactions on Power Systems,~Vol.~, No.~, September~2018}%
 {Shell \MakeLowercase{\textit{et al.}}}%

\maketitle

\begin{abstract}
We propose a scalable optimization framework for estimating convex inner approximations of the steady-state security sets. The framework is based on Brouwer fixed point theorem applied to a fixed-point form of the power flow equations. It establishes a certificate for the self-mapping of a polytope region constructed around a given feasible operating point. This certificate is based on the explicit bounds on the nonlinear terms that hold within the self-mapped polytope. The shape of the polytope is adapted to find the largest approximation of the steady-state security region. While the corresponding optimization problem is nonlinear and non-convex, every feasible solution found by local search defines a valid inner approximation. The number of variables scales linearly with the system size, and the general framework can naturally be applied to other nonlinear equations with affine dependence on inputs. Test cases, with the system sizes up to $1354$ buses, are used to illustrate the scalability of the approach. The results show that the approximated regions are not unreasonably conservative and that they cover substantial fractions of the true steady-state security regions for most medium-sized test cases.
\end{abstract}


\section{Introduction} \label{sec:intro}
The AC optimal power flow (ACOPF) representation of a power system forms a foundation for most of the normal and emergency decisions in power systems. The traditional ACOPF formulation targets the problem of finding the most economical generation dispatch admitting a voltage profile that satisfies operational constraints. Independent System Operators solve the ACOPF in different contexts on multiple time intervals ranging from a year for planning purposes to 5 minutes for real-time market clearing. Since it was first introduced in 1962, the OPF problem has been one of the most active research areas in the power system community. Being an NP-hard problem, it still lacks a scalable and reliable optimization algorithm \cite{cain2012history}, although last years were marked by tremendous progress in this area \cite{jabr2006radial, LavaeiLowzerogap, coffrin2016qc, cui2017new}.

Active adoption of renewable generation exposes the power grids to unprecedented levels of uncertainty. Unexpected variations of wind may push the operating point found by the OPF into technically unacceptable regions, or, in more dramatic scenarios, outside of the loadability regions where the power flow solution exists at all. Traditionally, the ability of the system to survive these events can be characterized by constructing the so-called steady-state security sets in power injection space. In this work we focus on the construction of convex inner approximations of the sets, that can naturally be used to assess the robustness of a given operating point to uncertainties in renewable or load power fluctuations. They can also be used to enforce feasibility in situations where constraining the system to nonlinear power flow equations and operational specifications is not practical. This may occur, for example, in the context of real-time corrective emergency control where the computation time is critical or in high-level optimization that is too complex to incorporate nonlinear constraints. Similarly, the reduced size inner approximations can be enforced in situations where instead of the regular feasibility, an additional margin is required from the solutions concerning load/renewable uncertainty. Whenever the approximation has a special geometric structure, it can also be used for decentralized decision making, for example, by allowing for independent and communication-free redispatch of power resources in different buses/areas. 

Like many other power system problems, the original work on inner approximations of power flow steady-state security sets was introduced by Schweppe and collaborators in the late 70s \cite{Schweppe1975}. The first practical algorithms based on fixed point iteration appeared in the early 80s \cite{wu1982steady}. In the Soviet Union, the parallel effort focused on the problem of constructing solvability sets for static swing equations \cite{Vasin1985, Byc1989}. More recently, new algorithms based on different fixed point iterations have been proposed for radial distribution grids without PV buses \cite{bolognani2016existence, EssieHungKostya, wang2016existence, nguyen2017framework}, decoupled power flow models equivalent to resistive networks with constant power flows \cite{simpson2016voltage} and lossless power systems \bstctlcite{IEEEexample:BSTcontrol} \cite{simpson2017part1, simpson2017part2}. Although more general approaches that do not rely on special modeling assumptions have been proposed in the literature \cite{DjTuritsyn2015Feasibility, Dvijotham2018LCSS, ChiangOPF}, they still suffer from either poor scalability or high conservativeness or both. It is also worth noting that this work is a natural evolution of the ideas originally presented in \cite{Dvijotham2018LCSS}; the approaches presented here and in \cite{Dvijotham2018LCSS} are in some sense complementary to each other. The techniques developed in \cite{Dvijotham2018LCSS} are applicable to general quadratic systems and come with explicit estimates of the conservativeness. However, they require norm-based estimates of Lipschitz constants, which tend to be overly conservative for large-scale systems due to the high conditioned numbers involved. On the other hand, this work does not rely on Lipschitz constant estimation but proposes a natural adaptation of the self-mapping region to the geometry of the problem which helps remove the conservativeness introduced in \cite{Dvijotham2018LCSS}.

In this work, we develop sufficient conditions for the existence of feasible power flow solutions. Based on these conditions, we propose a novel algorithmic approach to constructing inner approximations of steady-state security domains that can apply to the most general formulation of power flow equations without any restrictions on the network and bus types. The size of the resulting regions is comparable to the actual feasibility domain even for large-scale models and can be potentially improved in the future. In comparison to all related approaches above, the inner approximations developed in this work are more flexible and adaptive as one can customize their construction for a number of specific preferences or tasks, by solving auxiliary optimization problems like those defined in \eqref{eq:optimalextension}. These optimization problems, although non-convex, allow for fast construction of the inner approximations. Any feasible solution to these problems, even a suboptimal one, establishes a region where the solution of the original power flow is guaranteed to exist and satisfy feasibility constraints. Performance of the proposed algorithm is validated on several medium- and large-sized IEEE test cases.

\section{Notation and an OPF formulation} \label{sec:OPFform}
We start this section by introducing the set of notation which will be used in this paper. We use $\Com$ to denote the set of complex number. $\conj{\pmb{x}}$ denotes the conjugate of $\pmb{x}\in\mathbb{C}^n$. The considered system is characterized by a graph with vertex set $\mathcal{V}$. There are three kinds of vertices in the system: load and generator sets denoted by $\mathcal{L}$ and $\mathcal{G}$, respectively, and the slack bus nodes denoted as $\mathcal{S}$. Further, $\mathcal{E}$ is a set of directed edges $ e = (k,l)$, $k, l \in \mathcal{V}$. The $\odot$ operator stands for element-wise product $[x\odot y]_i=x_iy_i$ for $x,y \in \mathbb{R}^n$.
In this work, we consider a traditional ACOPF-like formulation defined by the following equations and inequalities:
\begin{align} \label{eq:pf}
p_k + j q_k 
&= \sum_l V_k \overline{Y}_{kl} V_l \exp(-j\theta_e) , \quad k \in (\mathcal{L}, \mathcal{G})\\
V_k^{\min} &\leq V_k \leq V_k^{\max}, \quad k \in \mathcal{L} \label{eq:v} \\
\theta_{e}^{\min} &\leq \theta_{e} \leq \theta_{e}^{\max}, \quad e \in \mathcal{E} \\
p_{k}^{\min} &\leq p_{k} \leq p_{k}^{\max}, \quad k \in \mathcal{G}  \label{eq:pgen}\\
q_{k}^{\min} &\leq q_{k} \leq q_{k}^{\max}, \quad k \in \mathcal{G} \label{eq:qgen}\\
|s_e^f| &\leq s_e^{ max} , \quad |s_e^t| \leq s_e^{\max}, \quad e \in \mathcal{E}. \label{eq:imax}
\end{align}
Each bus $k\in \mathcal{V}$ of the system is characterized by a complex voltage $v_k = V_k \exp(j\theta_k)$ and apparent power injection $s_k = p_k + j q_k$. For each edge or line $e$ connecting bus $k$ and bus $l$, we define an angle difference $\theta_e = \theta_k - \theta_l$. The apparent power transfer $s_e$ through the line $e$ can be defined either at the beginning (on the ``from'' bus) or at the end (on the ``to'' bus) of the line, with the corresponding values denoted by the superscripts $f$ and $t$.

In the following derivations of the general results, we use $u$ to denote the vector of ``parameters'' $u = [p_\mathcal{G};p_\mathcal{L}; q_\mathcal{L}; V_\mathcal{G}]$ and use $x$ to denote the space of variables, i.e. $x =[\theta_\mathcal{G};\theta_\mathcal{L};V_\mathcal{L}; q_\mathcal{G}]$. This distinction reflects the typical situation in power systems, where the components of $u$ can be either controlled directly by the system operators, or change in response to external factors, while the components of $x$ are usually not directly controllable and are determined by the physical power flow equations. Generally, the choice of parameter and variable vectors depends on the context and may be different in some other applications.

In this work, we focus on the characterization of the steady-state security set. This is the set of control inputs $u$ for which the power flow equations \eqref{eq:pf} has at least one feasible solution which satisfies all the operational constraints \eqref{eq:v} to \eqref{eq:imax}.

The steady-state security set is generally non-convex and may not even be single connected. In this work, we construct convex inner approximations of polytopic shape for this set. Mathematically, they represent sufficient conditions for the existence of feasible solution that can be quickly checked or enforced in optimization procedures where the optimality of the power flow solution is not critical. For example, they can be used for online identification of viable load-shedding actions that are guaranteed to save the system from voltage collapse events.

\section{The feasibility of input-affine systems} \label{sec:affinput}

In this section, we develop our general approach for deriving sufficient conditions for the existence of a feasible solution to a nonlinear constrained system of equations. Our approach applies to a broader class of nonlinear systems that satisfy two essential properties: they depend affinely on the inputs (parameters), and their nonlinearity can be expressed as a linear combination of nonlinear terms that admit simple bounds. The derived results for these general nonlinear systems are applied to the ACOPF problem in Section \ref{sec:OPFfeas}.

Consider a following set of nonlinear equations on $x \in \mathbb{R}^n$ that depend affinely on the input vector $u \in \mathbb{R}^k$:
\begin{equation} \label{eq:abstract}
 M f(x) =  M f(x^\star) + R (u - u^\star).
\end{equation}
Here, $x^\star$ is the base solution of the problem \eqref{eq:abstract} which corresponds to the base parameter vector $u^\star$, $M\in \mathbb{R}^{n\times m}$,  $R \in \mathbb{R}^{n \times k}$, and $f=[f_1;f_2;\ldots;f_m]:\mathbb{R}^n \to \mathbb{R}^m$ is the set of nonlinear ``primitives'', i.e. simple nonlinear terms usually depending on few variables. In the power system context, the problem \eqref{eq:abstract} represents the power flow equations \eqref{eq:pf} or \eqref{eq:ySx}; the base solutions is typically the base operating point which is assumed to be feasible. Feasibility conditions is represented by a set of linear and nonlinear constraints with the latter represented in the form
\begin{align}
     h(x)    &\leq 0. \label{eq:feasconst}
\end{align}
Without any loss of generality, we assume that $h(x): \mathbb{R}^n \to \mathbb{R}^l$ can be expressed in terms of the same nonlinear primitives function $f(x)$, i.e., $h(x)=T f(x)$ with $T\in\mathbb{R}^{l\times m}$. Also, we assume that $h(x^\star) \leq 0$ in our construction framework. For the ACOPF problem, the constraint \eqref{eq:feasconst} is a compact representation of the set of nonlinear constraints \eqref{eq:qgen}-\eqref{eq:imax}; the corresponding forms of matrix $T$ to these constraints will be discussed later in Section \ref{sec:iqconst}. The Jacobian of the corresponding system of equations is defined as
\begin{align}
J(x)=M
\left.\frac{
\partial f}{\partial x}
\right|_x,\, [J(x)]_{ij}=\sum_{k=1}^m M_{ik} \frac{\partial f_k}{\partial x_j},\label{eq:Jac}
\end{align}
where $ J_\star = J(x^\star) $ refers to the Jacobian evaluated at the base solution that we assume to be non-singular. 

In the following paragraphs, we will derive a fixed-point representation of the system \eqref{eq:abstract} that allows for simple certification of solution existence. This solvability certificate is based on the bounds on nonlinearity in the neighborhood of the base operating point. Compact representation of the corresponding convex regions in the input, state, and nonlinear image spaces will be introduced below to facilitate our certificate development.

For any differentiable nonlinear map $f: \mathbb{R}^n \to \mathbb{R}^m$ and a base point $x^\star \in \mathbb{R}^n$, we define nonlinear ``residual'' maps $\delta f(x)$ and $\delta_2 f(x)$ as the combinations of $f(x)$, base value $f^\star = f(x^\star)$, and base Jacobian $ \partial f/\partial x|_{x^\star}$ via the following:
\begin{align}
 \delta f(x) &= f(x) - f^\star  \\
 \delta_2 f(x) &= f(x) - f^\star -  \frac{\partial f}{\partial x}\Big|_{x=x^\star}\cdot(x - x^\star).
\end{align}

The defined residual operator allows for compact representation of the original system \eqref{eq:abstract} in the fixed-point form:
\begin{equation} \label{eq:fixedpointfull}
 x =  x^\star + J_\star^{-1}R(u-u^\star) - J_\star^{-1} M \delta_2 f(x),
\end{equation}
or, even more compactly as 
\begin{equation} \label{eq:fixedpoint}
\xdiff = F(\xdiff;\udiff)= J_\star^{-1} R \udiff + \phi(\xdiff).
\end{equation}

Here we use $\xdiff  = x - x^\star$ and $\udiff  = u - u^\star$ to denote the deviations of the variable and input vector from the base point. $\phi(\xdiff) = - J_\star^{-1} M \delta_2 f(x^\star + \xdiff)$ represents the nonlinear corrections. It is important to note that the Jacobian $J$ is not $\partial f/\partial x$, but the one defined in \eqref{eq:Jac}. 

Next, we assume that the input deviations $\udiff$ belong to a structured box-shaped uncertainty set $\mathcal{U}$ defined as $ -\limu^-\leq \udiff \leq \limu^+$, or more compactly as $ \mathcal{U}(\limu) = \left\{\udiff \in \mathbb{R}^k|\, {\pm \udiff  \leq \limu^\pm}\right\}$. Also, we use $\limu$ to denote $\limu=[\limu^-;\limu^+]$.

Having the fixed point form \eqref{eq:fixedpoint}, we will apply the Brouwer's fixed point theorem (Chapter 4, Corollary 8 in \cite{spanier1994algebraic}) stated below to derive a sufficient condition for solvability.

\begin{theorem}[Brouwer's fixed point theorem]
Let $F: \mathcal{X} \mapsto \mathcal{X}$ be a continuous map where $\mathcal{X}$ is a compact and convex set in $\mathbb{R}^n$. Then the map has a fixed point in $\mathcal{X}$, namely $F\br{x}=x$ has a solution in $\mathcal{X}$.
\end{theorem}

In this work, we propose to use a parameterized polytope representation of $\mathcal{X} = \mathcal{A}(\limy)$ with a fixed face orientation but varying distances to individual faces. We refer to this set as an ``admissibility polytope''.

Formally, the admissibility polytope family $\mathcal{A}(\limy)$ defines the following non-empty polytope for any non-zero matrix of bounds $\limy = [-\limy^-, \limy^+] \in\mathbb{R}_+^{n\times 2}$:
\begin{equation} \label{eq:Alimy}
 \mathcal{A}(\limy)  = \{\xdiff \in\mathbb{R}^n| \pm A \xdiff \leq \limy^\pm \}
\end{equation}
where $\pm A \xdiff \leq \limy^\pm$ is shorthand for the following:
\begin{align*}
\{A \xdiff \leq  \limy^+, 
-A \xdiff \leq \limy^-\}
\equiv -\limy^- \leq A \xdiff \leq \limy^+.
\end{align*}

The polytope $\mathcal{A}(\limy)$ will be used to establish the neighborhoods of the base operating point where the solution is guaranteed to exist. Extra conditions of the form $\limy \leq \limy^{\max}$ can be then used to impose the linear ``feasibility'' constraints, such as voltage and line angle difference bounds \eqref{eq:v} -- \eqref{eq:pgen}. The choice of the matrix $A$ has a critical effect on the performance of our algorithm and will be discussed in more detail later.

To prove the existence of the solution upon choosing matrix $A$, we need to show the self-mapping of this polytope under the nonlinear map $F(\xdiff;\udiff)$. The most critical step in certifying the self-mapping is the establishment of nonlinearity bounds. Specifically, we assume that whenever the system is inside the admissibility polytope, i.e., $\xdiff\in\mathcal{A}(\limy)$, the nonlinear residual terms $\delta_2 f \in \mathbb{R}^m$ can also be naturally bounded based explicitly on $\limy$. These bounds are constructed explicitly for the power flow problem in the later parts of the paper. For a more general case, we formalize them by introducing the ``bounding polytope family''  $\mathcal{N}(\limy)$ that certifies the following bounds on the nonlinear map $f: \mathbb{R}^n \to \mathbb{R}^m$:
\begin{equation}
\xdiff \in \mathcal{A}(\limy) \Longrightarrow \delta_2 f(x^\star + \xdiff ) \in \mathcal{N}(\limy).
\end{equation}
In other words, the nonlinear bounding polytope family establishes a bound on the nonlinear function given the bounds on its argument. In this work, we rely on the simple, though not particularly tight, bounds illustrated on Figure \ref{fig:nonlinearbound} and formally represented by the functions $\delta_2f_k^\pm(\limy)$
\begin{equation} \label{eq:Nlimy}
 \mathcal{N}(\limy) =  \{z\in \mathbb{R}^m|\, {\pm z_k} \leq \delta_2f_k ^\pm(\limy)\}.
\end{equation}

This condition can be also considered as a formal definition of the function $\delta_2 f_k^\pm(\limy)$. 
\begin{remark} \label{rem:1} By definition, the functions $\delta_2 f_k(x)$ have zero derivatives at $x= x^\star$. Thus, for continuous and smooth maps $f$, there exists a nonlinear bound $\delta_2 f_k^\pm(\limy)$ such that $\|\delta_2 f^\pm_k(\limy)\|/\|\limy\| \to 0$ as $\limy \to 0$. In other words, the magnitude of the nonlinear residuals can be bounded by a superlinear function in the vicinity of the the operating point.
\end{remark}

Given the definitions of the three convex polytopes $\mathcal{U}(\limu)$, $\mathcal{A}(\limy)$, and $\mathcal{N}(\limy)$, we can restate the Brouwer fixed point theorem in the following form:
\begin{figure}[t]
    \centering
    \includegraphics[width=0.9\columnwidth]{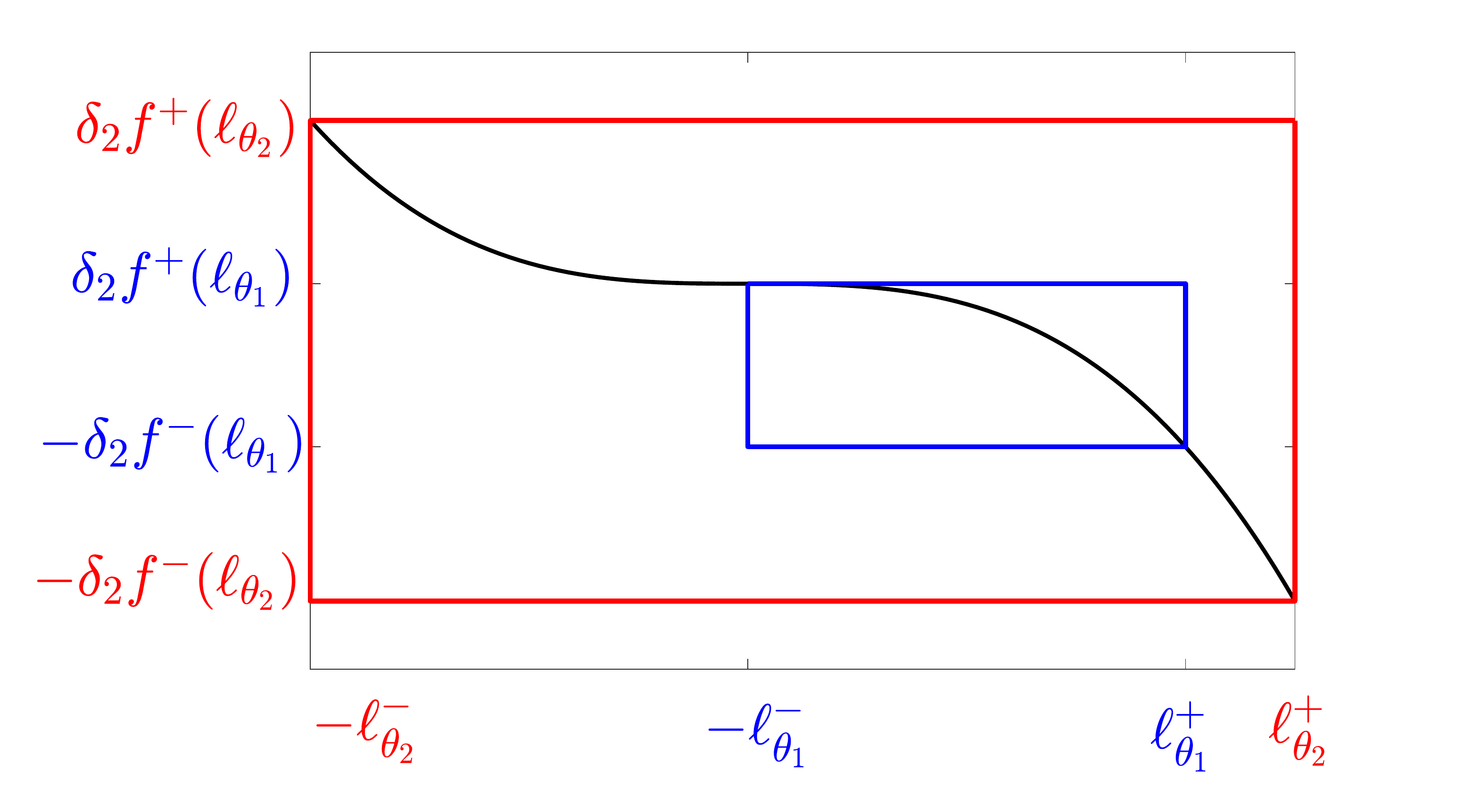}
	\caption{Nonlinear bounds for $\delta_2\{\sin(\theta)\} = \sin(\theta)-\theta$ for two values of $\ell_\theta$: $\ell_{\theta_1} = [0, 0.4]$ and $\ell_{\theta_2} = [0.4, 0.5]$}
		\label{fig:nonlinearbound}
\end{figure}
\begin{corollary}(Self-mapping condition) \label{cor:selfmaptheo} 
Suppose there exist bounds $\limy, \limu$ such that 
\begin{align}
\xdiff \in \mathcal{A}(\limy), \,\udiff \in \mathcal{U}(\limu) \implies F(\xdiff;\udiff) \in \mathcal{A}(\limy). \label{eq:selfmaptheo}    
\end{align}
Then \eqref{eq:fixedpoint} is solvable for all $\udiff \in \mathcal{U}(\limu)$ and has at least one solution inside admissibility polytope: $\xdiff \in \mathcal{A}(\limy)$. \end{corollary}

In Appendix \ref{app:validatecons}, we derive explicit condition on $\limy, \limu$ that guarantee satisfaction of \eqref{eq:selfmaptheo}. The condition allows us to formulate the following optimization problem which is a central result of this paper. 

\textbf{Construction of optimal inner approximation.} 
Given the definitions introduced above, assume that the quality of the inner approximation is given by the cost function $ c(\limy, \limu)$. Then the best approximation can be constructed by solving the following optimization problem:
 \begin{align}
 \label{eq:optimalextension}
 & \max_{\limy, \limu} \, c(\limy, \limu)\quad \mathrm{subject\, to:} \\
 & \limy^\pm \geq B^+ \limu^\pm + B^- \limu^\mp +C^+ \delta_2 f^{\pm}(\limy) + C^{-}\delta_2 f^{\mp}(\limy)\nonumber\\
 & \limy \leq\limy^{\max} \nonumber\\
 & D^+ \limu^+ + D^- \limu^- + E^+\delta_2 f^+(\limy) + E^-\delta_2 f^-(\limy) + h_\star \leq 0. \nonumber
\end{align}

Here $h_\star = h(x_\star)$ and the matrices $B, C ,D, E $ are given by $B= A J_\star^{-1} R$, $C=- A J_{\star}^{-1} M$, $D = TLJ_\star^{-1} R$, $E = T - TLJ_\star^{-1} M$, and $L = \partial f/\partial x|_{x_{\star}}$. The upper indices $^\pm$ on the matrices refer to the absolute values of the positive and negative parts of these matrices, i.e., for any matrix $F$, $F = F^+ - F^-$ with $F^{\pm}_{ij} \geq 0$.

Problem \eqref{eq:optimalextension} is a non-convex and nonlinear program in the space of uncertainty and admissibility polytope shapes $\limu, \limy$. Although the optimal solution to this problem cannot be reliably found in polynomial time, any feasible solution to \eqref{eq:optimalextension} establishes a certified inner approximation of the steady-state security set. Therefore, the natural strategy is to solve this problem is based on local optimization schemes. The local search is guaranteed to produce some feasible certificate because sufficiently small values of $\limy$ are always feasible. This follows from the superlinear behavior of $\delta_2f^\pm(\limy)$ in the neighborhood of $\limy = 0$ as discussed in Remark \ref{rem:1}.

Due to the nonlinear and non-convex nature of the optimization problem, the quality of the resulting certificate depends critically on the initial guess used for initialization of the local search. In our experiments, we have found that one viable strategy of finding good initialization points lies in using the solution to the linear relaxation of \eqref{eq:optimalextension} to initialize the nonlinear local search. The idea behind the relaxation based approach to initialization is very simple. The main source of nonlinearity in \eqref{eq:optimalextension} is the $\delta_2 f^\pm (\limy)$ terms entering the self-mapping condition and nonlinear feasibility constraints. These constraints can be relaxed into linear ones by considering linear upper bounds on $\delta_2 f^\pm(x)$. Obviously, any feasible solution to this linear problem satisfies the original problem as well and establishes a valid certificate. This certificate is generally more conservative than the one obtained by solving a full nonlinear problem as the bounds on the nonlinearity are less tight. Moreover, the linear nature of the relaxed constraints invalidates Remark \ref{rem:1} guaranteeing the existence of some feasible solution. Indeed, the resulting set of linear constraints may be infeasible, although this does not happen in our experiments. On the other hand, the efficiency of modern linear programming solvers allows for an extremely fast construction of some polytope that can be later improved via local nonlinear search. The details of this relaxation are presented in Appendix \ref{app:ext}.

The objective function $c(\limy, \limu)$ can be tailored to a specific use case of the resulting polytope. Below we present three possible choices for the certificate quality function $c(\limy, \limu)$  for common problems arising in power system analysis.

\subsection{Natural certificate quality functions $c(\limy, \limu)$}
\subsubsection{Maximal feasible loadability set}
A classical problem arising in many domains, including power systems, is to characterize the limits of system loadability characterized by a single stress direction in injection space. In this case, the only column of the matrix $R$ defines the stress direction, while the positive scalar $u$ defines the feasible loadability level.  The goal of the feasible loadability analysis is to find the maximal level $u$ for which the solution still exists and is feasible. According to Corollary \ref{cor:ubox}, all $u$ satisfying $ u^\star -\limu^- \leq u \leq  u^\star + \limu^+$ are certified to have a solution, hence the objective function of \eqref{eq:optimalextension} is simply $c(\limy, \limu) = \limu^+$. Of course, the simple loadability problem can be more easily solved using more traditional approaches like the continuation method.

\subsubsection{Robustness certificate} \label{sec:robustness}
In another important class of applications, the goal is to characterize the robustness of a given solution concerning some uncertain inputs. In power system context, it could be the robustness concerning load or renewable fluctuations. Assume, without the loss of generality, that the input variable representation is chosen in a way that all the components of $\udiff$ are uniformly uncertain with uncertainty set centered at zero. In this case, the goal is to find the largest value of $\lambda$ such that the feasible solution exists for any $\udiff$ satisfying $|\udiff_k| \leq \lambda$ for all $k$. This problem can be naturally solved by maximizing $c(\limy, \limu) = \lambda$ with an extra set of uniformity constraints $[\limu^\pm]_k = \lambda$ ensuring that the box $\mathcal{U}$ represents a symmetric cube.

It is worth mentioning that the problem of assessing the robustness of a power system under renewable fluctuations is closely related to the ``dispatchable region'' problem considered in \cite{6949701, 7079527}, and also discussed under the notion of Do-Not-Exceed limits in \cite{6856230,7024953,7453153}. Our technique allows the same kind of limits to be constructed without relying on any linearized power flow approximations utilized in these previous works.

\subsubsection{Chance certificates}
In another popular setting, one assumes some probability distribution of the uncertain inputs and aims to find a polytope that maximizes the chance of a randomly sampled input being certified to have a solution. This problem can also be naturally represented in the generic certificate optimization framework. Assume, without major loss of generality, that the inputs are i.i.d. normal variables. In this case, the log-probability $\mathbf{P}$ of a random sample falling inside the box $\mathcal{U}(\limu)$ is given by 
\begin{equation} \label{eq:chancer}
    \log \mathbf{P}(\limu) = \frac{1}{2}\sum_k \mathrm{erf}\left(\frac{[\limu^+]_k}{\sqrt{2}}\right) - \mathrm{erf}\left(-\frac{[\limu^-]_k}{\sqrt{2}}\right).
\end{equation}
In \eqref{eq:chancer}, $\mathrm{erf}()$ is the ``error function'' or ``Gauss error function'' used in probability theory. Its interpretation is as the following: given a random variable $Z$ normally distributed with mean $0$ and variance $1/2$, $\mathrm{erf}(x)$ represents the probability of $Z$ falling in the range $[-x, x]$ \cite{andrews1992special}. Then the maximal chance certificate is established by maximizing $c(\limy, \limu) = \log \mathbf{P}(\limu)$ within the central optimization problem \eqref{eq:optimalextension}.

\subsection{Selection of matrix $A$} 
Choice of the matrix A has an enormous effect on the performance of the approach and the quality of the resulting certificates. The implications of the choice of $A$ are two-fold: first it determines the shape of the region that is certified to be self-mapped by the fixed representation of the equations. And second, the conditions imposed by the constraint $\pm Ax \leq \limy^\pm$ are used for bounding the nonlinearities and verification of the self-mapping. Although more research is needed to determine the relative importance of these two effects caused by the choice of the matrix $A$, generally one could expect that the effect on the nonlinearity bounds is more pronounced. Hence, it is recommended to choose the matrix $A$ in a way that allows for effective nonlinear bounds. The matrix $A$ should be “compatible” with the representation of the nonlinear terms. In other words, the inequalities $\pm A \xdiff \leq \limy^\pm$ should allow for straightforward and non-conservative bounds of the individual nonlinear terms. Given the bounding approach employed in this work, the least conservative bounds are obtained when $A$ bounds individual variables entering in the nonlinear expressions. However, more general choices of $A$ can be possible for other nonlinearities encountered in different representations. As a side note, addition of extra rows to the matrix $A$ that are not utilized in bounding of nonlinearity will only complicate the certification of the self-mapping and will not improve the resulting polytope.

\section{Steady-state security} \label{sec:OPFfeas}
In this section, we apply the framework developed in previous sections to the power flow steady-state security problem. There are many possible ways to represent the power flow equations in the form \eqref{eq:abstract} amenable to the algorithm application. Moreover, the representation of the equation and the choice of nonlinear functions and their variables has a dramatic effect on the size of the resulting regions. We have experimented with a variety of formulations, in particular, traditional polar and rectangular forms of power flow equations with different choices of matrix $M$ and function $f$ in \eqref{eq:abstract}. Our experiments have indicated that the choice of the representation, accuracy of the nonlinear bounds, and the quality of the initial guess all strongly affect the conservativeness of the resulting certificates. A more thorough analysis of these factors is required to identify the most effective strategies of applying the algorithm to the power flow problem. These strategies will be discussed in details in our future work. In this paper, we present only one formulation that resulted in the least conservative regions for large-scale systems in our experiments. This formulation is based on the admittance representation of the power flow equations and nonlinear terms associated with power lines. It can naturally deal with strong (high admittance) power lines in the system that we suspect to be one of the sources of conservativeness in other more conventional formulations.

\subsection{Admittance based representation of power flow} \label{sec:admrep}
The power formulation discussed below is based on a non-traditional combination of node-based variables and branch-based nonlinear terms. This representation is naturally constructed using the weighted incidence type matrices. Specifically, we use $y^d \in \mathbb{C}^{|\mathcal{V}|}$ to represent the diagonal of the traditional admittance matrix, and matrices $Y^{f}, Y^t \in \mathbb{C}^{|\mathcal{V}|\times |\mathcal{E}|}$ as weighted incidence matrices representing the admittances of individual elements with $Y^f$ corresponding to the negative admittance of lines starting at a given bus, and $Y^t$ corresponding to the admittance of lines ending at it. Any power line with admittance $y_e$ connecting the bus $\mathrm{from}(e)$ to the bus $\mathrm{to}(e)$ contributes to the two elements in the matrices $Y^f$ and $Y^t$, specifically, $Y^f_{\mathrm{from(e)},e} = - y_e$ and $Y^t_{\mathrm{to}(e),e} = y_e$. In this setting, for a bus $k$ consuming the complex current $i_k$, the Kirchoff Current Law takes the form
\begin{equation} \label{eq:ik}
 i_k = y^d_k v_k + \sum_e Y^f_{ke} v_{\mathrm{to}(e)} + Y^t_{ke} v_{\mathrm{from}(e)}. 
\end{equation}
Whenever the PQ bus is considered with constant complex power injection $s_k$, one also has $\overline{i_k} = s_k/v_k$. Next, we introduce the logarithmic voltage variables as $\rho_k = \log (V_k)$, $\rho_e = \rho_{\mathrm{from}(e)} -\rho_{\mathrm{to}(e)}$, and rewrite the power flow equations in the following form:
\begin{align} \label{eq:admform}
 y_\mathcal{V} - y^d_\mathcal{V} &= Y^t_{\mathcal{V}\mathcal{E}} \mathrm{e}^{\rho_\mathcal{E} + j \theta_\mathcal{E}} + Y^f_{\mathcal{V}\mathcal{E}}\mathrm{e}^{-\rho_\mathcal{E} - j \theta_\mathcal{E}}
\end{align}
where $y_k = \overline{s_k}/|v_k|^2$. Assuming that the base solution is given by $\rho^\star, \theta^\star$ the equations can be rewritten as
\begin{equation}
y_\mathcal{V} - y^d_\mathcal{V} =  \hat{Y}^t_{\mathcal{V}\mathcal{E}} \mathrm{e}^{\delta \rho_\mathcal{E} + j \delta\theta_\mathcal{E}} + \hat{Y}^f_{\mathcal{V}\mathcal{E}}\mathrm{e}^{-\delta \rho_\mathcal{E} - j \delta \theta_\mathcal{E}}
\end{equation}
with $\hat{Y}^t = Y^t [\![ \exp(\rho_\mathcal{E}^\star + j \theta_\mathcal{E}^\star) ]\!]$ and $\hat{Y}^f = Y^f [\![ \exp(-\rho_\mathcal{E}^\star - j \theta_\mathcal{E}^\star)]\!]$ where the notation $[\![\cdot ]\!]$ is used to define the diagonal matrices. This equation can be further simplified to 
\begin{equation}
y_\mathcal{V} - y_\mathcal{V}^d =  Y^+_{\mathcal{V}\mathcal{E}} \cosh(\delta \rho_\mathcal{E} + j \delta\theta_\mathcal{E}) + Y^-_{\mathcal{V}\mathcal{E}}\sinh(\delta \rho_\mathcal{E} + j \delta \theta_\mathcal{E})
\end{equation}
with the help of $Y^\pm = \hat{Y}^t\pm\hat{Y}^f$. Finally, using $y - y^d =  g + j  b$ we obtain
\begin{equation} \label{eq:gbv}
\begin{bmatrix}
g_\mathcal{V} \\
b_\mathcal{V}
\end{bmatrix} = 
\begin{bmatrix}
\mathrm{Re}(Y^+) & \mathrm{Re}(Y^-) & 
-\mathrm{Im}(Y^-) & -\mathrm{Im}(Y^+)\\
\mathrm{Im}(Y^+) & \mathrm{Im}(Y^-) & 
\mathrm{Re}(Y^-)  & \mathrm{Re}(Y^+) 
\end{bmatrix} f(x)\nonumber
\end{equation}
where 
\begin{equation}
f(x)= \begin{bmatrix}
 \cosh\delta\rho_\mathcal{E}\odot\cos\delta\theta_\mathcal{E} \\
 \sinh\delta\rho_\mathcal{E}\odot\cos\delta\theta_\mathcal{E} \\
 \cosh\delta\rho_\mathcal{E}\odot\sin\delta\theta_\mathcal{E} \\
 \sinh\delta\rho_\mathcal{E}\odot\sin\delta\theta_\mathcal{E}
\end{bmatrix}.
\end{equation}

The nonlinear vector $f(x)$ consists of element-wise products of the hyperbolic and trigonometric functions. To find the bounds $\delta_2 f^\pm$ which are required in the steady-state security set construction \eqref{eq:optimalextension}, we apply nonlinear bounding introduced in Appendix \ref{app:nlbounds}. 

Like the traditional power flow equations, we use $x = (\theta_\mathcal{G}, \theta_\mathcal{L}, \rho_\mathcal{L})$ and consider only a subset of all the equations which are compatible with the general form \eqref{eq:abstract}:
\begin{align} \label{eq:ySx}
\begin{bmatrix}
g_\mathcal{G} \\
g_\mathcal{L} \\
b_\mathcal{L}
\end{bmatrix} &
\!= \!
\begin{bmatrix}
\mathrm{Re}(Y^+_{\mathcal{G}\mathcal{E}}) & \mathrm{Re}(Y^-_{\mathcal{G}\mathcal{E}}) & 
-\mathrm{Im}(Y^-_{\mathcal{G}\mathcal{E}}) & -\mathrm{Im}(Y^+_{\mathcal{G}\mathcal{E}})\\
\mathrm{Re}(Y^+_{\mathcal{L}\mathcal{E}}) & \mathrm{Re}(Y^-_{\mathcal{L}\mathcal{E}}) & 
-\mathrm{Im}(Y^-_{\mathcal{L}\mathcal{E}}) & -\mathrm{Im}(Y^+_{\mathcal{L}\mathcal{E}})\\
\mathrm{Im}(Y^+_{\mathcal{L}\mathcal{E}}) & \mathrm{Im}(Y^-_{\mathcal{L}\mathcal{E}}) & 
\mathrm{Re}(Y^-_{\mathcal{L}\mathcal{E}})  & \mathrm{Re}(Y^+_{\mathcal{L}\mathcal{E}}) 
\end{bmatrix}\! f(x).
\end{align}

We define the parameter vector $\, \udiff(V) = (p/V^2 - p^\star/(V^\star)^2, -q/V^2 + q^\star/(V^\star)^2)$ which consists of nodal admittance variations associated to the left hand side of \eqref{eq:admform}, respectively. Further, let $V^{\min} = V^\star - \delta V^-$ and $V^{\max} = V^\star + \delta V^+$ represent the minimum and maximum voltage levels associated with the voltage deviation bounds. Assume that the nonlinear term $\phi(\xdiff)$ is contained in the polytope $\mathcal{A}(\tau)$ where $\tau$ is defined in Lemma \ref{le:tau}. Then Corollary \ref{cor:ytop} below is essential to characterize the admissible range of nodal power injections for a given feasible range of admittance variations.

\begin{corollary}[Admissible power injection estimation] \label{cor:ytop}Given that $V^{\min} \leq V \leq V^{\max}$, if the active power injection satisfies
\begin{equation} \label{eq:prange}
p^{\min}\leq p \leq p^{\max},
\end{equation}
then the admittance deviation can be bounded as $ \pm \udiff \leq \limu^\pm$. Here $p^{\max} = (V^{\min})^2\left(\limu^+ +p^\star/(V^\star)^2\right)$, and 
\begin{equation}
    p^{\min} = \begin{cases}
    (V^{\min})^2\left(-\limu^- + \frac{p^\star}{(V^\star)^2}\right) & \text{if} -\limu^- +\frac{p^\star}{(V^\star)^2} < 0 \\
    (V^{\max})^2\left(-\limu^- +\frac{p^\star}{(V^\star)^2}\right) & \text{otherwise.}
    \end{cases}
\end{equation}
Similar bounds apply to the nodal reactive power.
\end{corollary}
\begin{proof}
Combining the definition of the admittance deviation $\udiff(V) = p/V^2 - p^\star/(V^\star)^2$ and the condition $\pm \udiff \leq \limu^\pm$, yields
\begin{equation}
V^2\left(-\limu^- + p^\star/(V^\star)^2\right) \leq p \leq V^2\left(\limu^+ +p^\star/(V^\star)^2\right). \nonumber
\end{equation}

Next, the admissible range of active power is given by $p^{\max} = \min_{V}V^2\left(\limu^+ +p^\star/(V^\star)^2\right)$ and $p^{\min} = \max_{V}V^2\left(-\limu^- +p^\star/(V^\star)^2\right)$. 
\end{proof}

\subsection{Thermal and reactive power constraints} \label{sec:iqconst}
Lemma \ref{lem:feasconst} establishes a way to incorporate the feasibility constraints of the form $T f(x) \leq 0$ where $f(x)$ contains the element-wise products of the hyperbolic and trigonometric functions that appear in \eqref{eq:ySx}. Both thermal power and reactive power generation constraints can be presented in this form; below we will show the corresponding coefficient matrix $T$.
\paragraph{Thermal power limit constraint} by following the sequence from \eqref{eq:ik} to \eqref{eq:gbv}, we can derive an equivalent admittance based expression for the power transfers. As the power transfers ``to'' and ``from'' for each line are very similar, we present only the ``from'' transmission.

The current on line $e = (k,j)$ can be represented as:
\begin{equation} \label{eq:ifrom}
 i_{e}^f = Y^f_{e} v_{\mathrm{to}(e)} + Y^t_{e} v_{\mathrm{from}(e)}. 
\end{equation}
where  $\mathrm{from}(e) = k$ and $\mathrm{to}(e) = j$. Let $y_e^f = \overline{(s_e^f)}/V_k^2$ and $y_e^f - Y_e^t = g_e^f + j b_e^f$. By repeating the admittance base derivation, we ultimately yield the following expression:
\begin{align} \label{eq:gbe}
\begin{bmatrix}
g_\mathcal{E}^f \\
b_\mathcal{E}^f
\end{bmatrix} &= 
\begin{bmatrix}
\mathrm{Re}(Y_\mathcal{E}^+) & \mathrm{Re}(Y_\mathcal{E}^-) & 
-\mathrm{Im}(Y_\mathcal{E}^-) & -\mathrm{Im}(Y_\mathcal{E}^+)\\
\mathrm{Im}(Y_\mathcal{E}^+) & \mathrm{Im}(Y_\mathcal{E}^-) & 
\mathrm{Re}(Y_\mathcal{E}^-)  & \mathrm{Re}(Y_\mathcal{E}^+) 
\end{bmatrix} f(x)
\end{align}
where $Y_\mathcal{E}^\pm = \pm \Delta_{\mathcal{V}\mathcal{E}}\hat{Y}_{\mathcal{V}\mathcal{E}}^f$, and the entry $\Delta_{k,e}$ becomes $1$ only when $k = \mathrm{from}(e)$ and is $0$ otherwise. Then the auxiliary matrix $T$ is easily obtained from \eqref{eq:gbe}.

Concerning the thresholds on the apparent power transfer, we propose to enforce intermediate constraints on the active and reactive power transfers, i.e., $ \ell_{p,e}^-\leq p_e \leq \ell_{p,e}^+$ and $ \ell_{q,e}^-\leq q_e \leq \ell_{q,e}^+$. These intermediate constraints can be easily expressed in the form of $h(x) \leq 0$ which later can be incorporated in the framework \eqref{eq:optimalextension} with the form of the last set of constraints. Then, the limit on the apparent power can be guaranteed by imposing additional conditions $(\ell_{p,e}^\pm)^2 + (\ell_{q,e}^\pm)^2 \leq (s_e^{\max})^2$.

\paragraph{Reactive power generation constraint} Like the load power and active power generation presented in \eqref{eq:ySx}, the reactive generation can be easily represented using an admittance form. The associated matrix $T$ is then given by a single block-row:
\begin{equation}
-\begin{bmatrix}
\mathrm{Im}(Y^+_{\mathcal{G}\mathcal{E}}) & \mathrm{Im}(Y^-_{\mathcal{G}\mathcal{E}}) & 
\mathrm{Re}(Y^-_{\mathcal{G}\mathcal{E}})  & \mathrm{Re}(Y^+_{\mathcal{G}\mathcal{E}}) 
\end{bmatrix}.
\end{equation}
The corresponding conditions are constructed in a manner similar to the thermal limits. We impose conditions on reactive power generation level by introducing intermediate constraints $\limq^- \leq q \leq \limq^+$. Then the reactive limits are enforced simply as $q^{\min} \leq \limq^- \leq \limq^+ \leq q^{\max}$.

It is worth noting that in realistic power systems the generators continue to operate even after reaching the reactive power limits, typically switching from the PV to PQ nodes. While it should be possible to model such a switching behavior by introducing additional nonlinearities in the problem, in this first version of the algorithm we ignore this effect. The best strategies for incorporating the switching effect will be discussed in our future works. In the scope of this work, the reactive power limits are used to define the boundaries of the security region.

\section{Simulations} \label{sec:simulation}
For our algorithm validation and simulations, we relied on transmission test cases included in the MATPOWER package and constructed the inner approximations for all the cases up to $1354$ buses. We also compare the cross-section of our constructed inner approximation to the actual steady-state security sets as defined in this work, to assess the conservativeness of the approach. The problem specifications and the details of the construction are presented below.

\begin{figure}[t]
    \centering
    \includegraphics[width=\columnwidth]{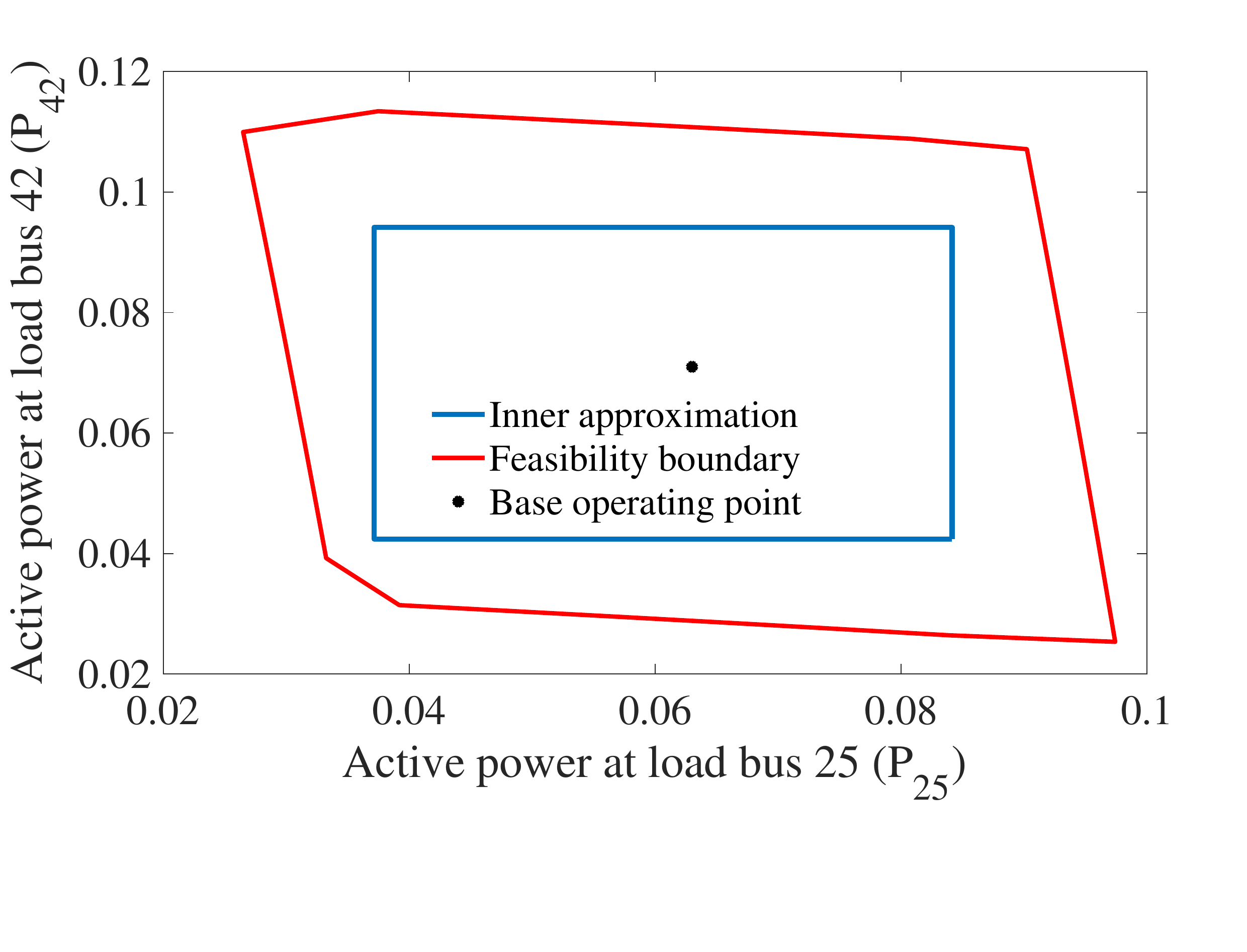}
	\caption{A steady-state security set of IEEE $57$-bus system}
		\label{fig:57}
\end{figure}

\subsection{Problem setup}
A typical constrained power flow problem, which is previously introduced in section \ref{sec:intro}, considers both power flow equality constraints \eqref{eq:pf}, and operational constraints (\ref{eq:v}-\ref{eq:imax}) including nodal voltage and angle limits, power generation limits, and thermal limits. While the voltage regulation standards require a voltage deviation of less than $\pm5\%$ from the nominal voltage level of $1 p.u.$, some base cases available in the MATPOWER package are not compliant with this requirement and have base operating points violating the constraints. In this work, we constrain the voltage to stay within $1\%$ around the base voltage level $x^\star$.

For the power generation limits, we adopt the data from the available MATPOWER test cases. The line thermal limits, however, are not always provided; thus for the sake of simplicity, we assume that the maximum power transfer level is double the corresponding value of the base operating condition, i.e., $s_e^{\max} = 2 |s_e^\star|$. The reactive power generation and thermal limits are considered for all test cases with less than $300$ buses. However, we exclude them while simulating $300$- and $1354$-bus systems as the actual feasibility margins become extremely small for the default base operating points.

\begin{figure}[t]
    \centering
    \includegraphics[width=\columnwidth]{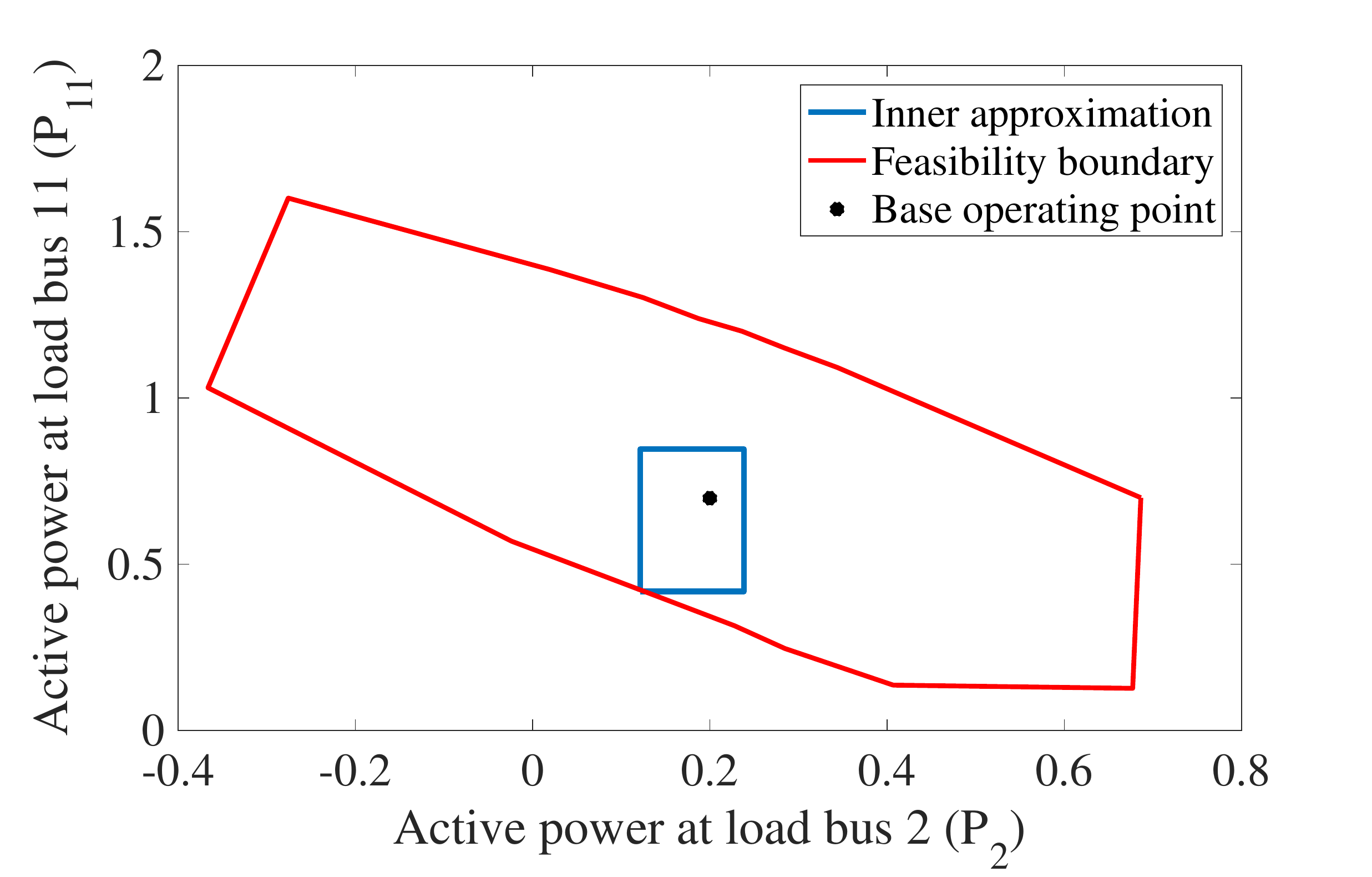}
	\caption{A steady-state security set of IEEE $118$-bus system}
		\label{fig:118}
\end{figure}

\begin{figure}[t]
    \centering
    \includegraphics[width=\columnwidth]{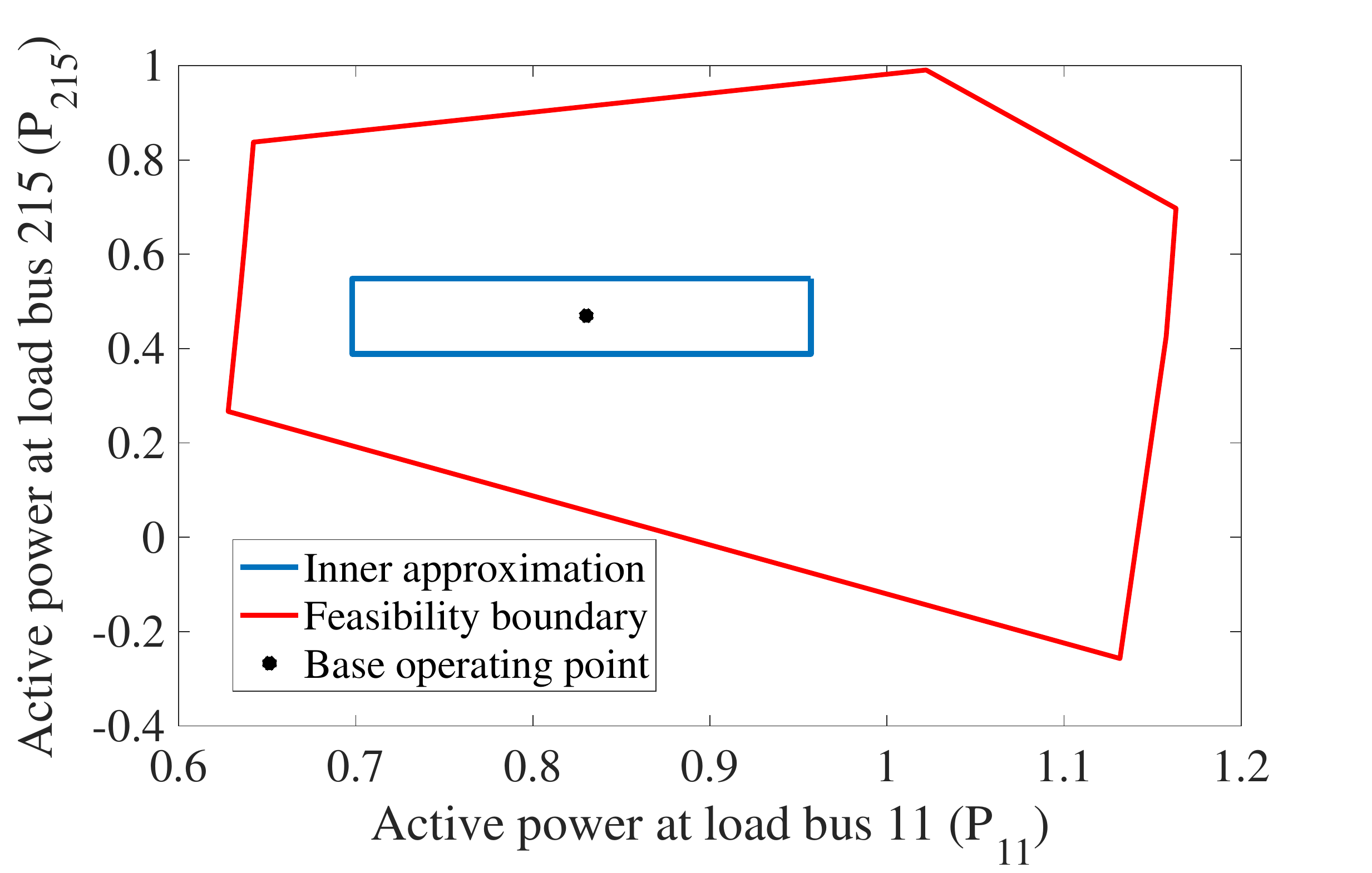}
	\caption{A steady-state security set of IEEE $300$-bus system}
		\label{fig:300}
\end{figure}

\subsection{Steady-state security region construction}
Upon setting up the OPF, we can solve problem \eqref{eq:optimalextension} to identify the ``optimal'' steady-state security sets. In the following simulations, we construct the optimal sets which maximize the uniformly boxed loading variations--the robustness certificate discussed in section \ref{sec:robustness}. For simplicity, we focus on the loading pattern $\udiff$ with only two load injection variations and visualize the estimated steady-state security sets in the corresponding 2D plane as shown in Figure \ref{fig:57} to Figure \ref{fig:pesage}.
Even though with the proposed framework we can estimate feasibility regions in all dimensions, visualizing the entire high-dimensional regions is impractical; thus, we plot only the cross sections on the plane of two nodal power injections. Concerning optimal certificates, we consider uniform bounds on the injection input perturbations by assuming that $\limu^-=\limu^+$; as a result, the estimated feasible sets in the admittance space of $\udiff$ are hypercubes. 

As discussed in section \ref{sec:affinput} and the Appendix,  the optimization of the large-scale test cases is initialized with the solution of the LP relaxation of the original problem. Although the LP optimal solution can be further improved by using nonlinear optimization solvers, for the largest test case considered here, the results by the LP procedure without additional improvement through the nonlinear local search are reported. This is done because of the exceptional performance of the modern LP solvers, which find an estimate in just tens of seconds in comparison to several hours it takes the (off-the-shelf and untuned) IPOPT solver to converge to the local minima. The slow performance of the IPOPT solver is likely related to implementation issues and should be resolved in future extensions of this work.

\begin{figure}[t]
    \centering
    \includegraphics[width=\columnwidth]{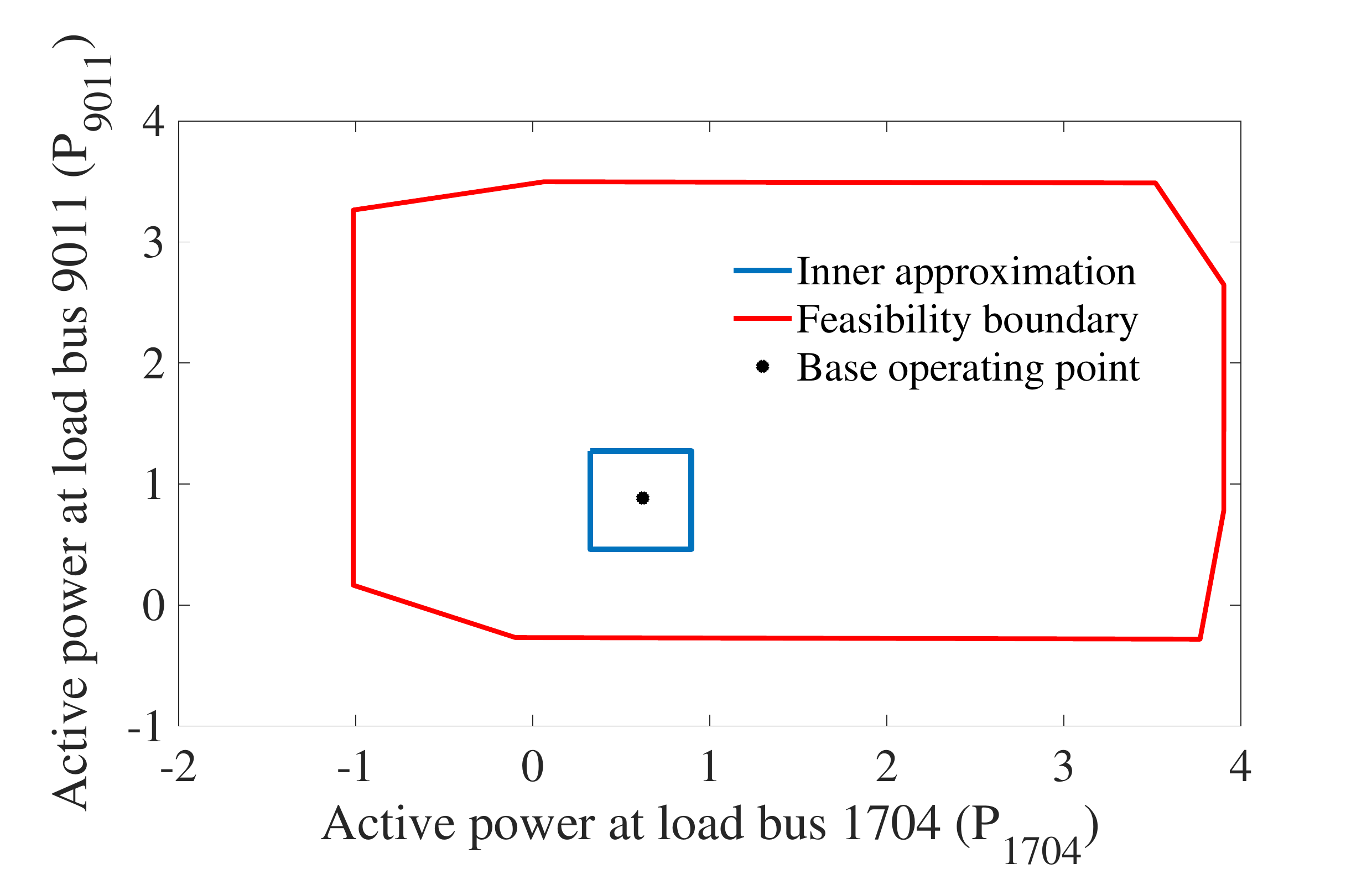}
	\caption{A steady-state security set of $1354$-bus medium part of European system}
		\label{fig:pesage}
\end{figure}

To evaluate the inner approximation technique, we compare the estimations with the actual steady-state security sets. Unlike the inner approximations constructed by our algorithm, which are given explicitly by a set of inequality relations, the actual boundary relies on a point-wise construction scheme. More specifically, we use Continuation Power Flow \cite{ajjarapu1992continuation} to trace the critical points along multiple loading directions and construct the actual boundary by connecting all these points.

\begin{table}[ht]
\centering
  \begin{tabular}{| c || c | c |}
    \hline
    System size  & Covering ratio & Tightness \\ \hline
    9  & 0.06 & 1\\ \hline
    39  & 0.4102 & 1\\ \hline
    57 & 0.53 & 0.833\\ \hline
    118 & 0.083 & 1\\ \hline
    300  & 0.13 & 0.645\\ \hline
    1354  & 0.036 & 0.335\\ \hline
  \end{tabular}
\caption{The performance of the inner approximation}
  \label{table:performance}
\end{table}

In most cases, as shown in Figures \ref{fig:57} to \ref{fig:pesage}, the approximated sets cover a substantially large fraction of the true steady-state security regions. Table \ref{table:performance} shows the covering ratio and the tightness -- the largest ratio between the estimated set and the actual set among all loading directions on the plotting plane-- for more test cases. In some cases like the IEEE $118$-bus test case plotted in Figure \ref{fig:118}, the tightness reaches its maximum value of $1$ when the estimated boundary ``touches'' the true limit along some direction. However, similar to the cause of the estimation conservativeness, the tightness conditions are not fully understood.

As the considered test systems are very different in terms of network properties like topology and line parameters, the results indicate that the proposed technique is stable and effective for a wide variety of transmission systems. For some large-scale cases like the $300$-bus system, the resulting approximations turned out to be conservative. The reasons for the conservativeness may be different; in particular, for $300$-bus case we believe the conservativeness was caused by the abnormally high reactive power compensation. Another common source of the conservativeness are the low impedance lines observed in some of the networks, which cause high sensitivity of the power flows to the angle/voltage variations. These lines can significantly affect the size of the resulting self-mapped polytope in the variable space and reduce the quality of the end certificates. The admittance based representation used in this paper alleviates some part of this problem, but does not eliminate the conservativeness of these cases entirely.

\section{Conclusion}
In this work we, developed a novel framework for constructing inner approximations of the steady-state security sets that does not rely on any special modeling assumptions and approximations and is applicable for general transmission system power flow problems. The framework is based on Brouwer fixed point theorem, which is applied to polytopic regions in variable (voltage-angle) space. In comparison to other studies of the same problem, our framework does not yield simple closed-form expressions for the region and instead relies on nonlinear optimization procedures to find the largest certifiable approximation of the steady-state security region and the corresponding polytope in the variable space with the self-mapping property.

The proposed framework can naturally be applied to a number of common settings, ranging from the classical loadability problem to robustness analysis and chance constraint certification. Such a nonlinear optimization problem is non-convex and likely NP-hard; however, any feasible solution, even not globally optimal one, is a valid certificate. Furthermore, the complexity of the problem scales linearly in the system size; valid certificates can be constructed with computational effort comparable to solving the original OPF problem. Simulation results show that the proposed technique is scalable and can provide sufficiently large estimations.

The quality of the approximation heavily depends on how tightly the nonlinear map \eqref{eq:fixedpointfull} can be bounded, and the results can be improved by using more adaptive bounds. Particularly, to approximate the nonlinear function, we can use more faced polytopes instead of simple hyper-rectangles like the ones plotted in Figure \ref{fig:nonlinearbound}. In general, the choice of matrix $A$ (in \eqref{eq:Alimy}) plays a critical role in the performance of the algorithm. Extending the matrix $A$ with extra rows may help in improving the algorithm, as long as the corresponding inequalities for the variables can be transformed in better bounds on the nonlinearity.  Moreover, as some nonlinear functions appearing in the fixed point map do not contribute significantly to the overall residuals, the corresponding arguments can be bounded in a more loose way allowing for easier certification of the self-mapping. A further line of improvement lies in coming with better initialization strategies and generally improving the performance of local search nonlinear solver.

The proposed framework can be extended to a number of other practically important cases. Of particular interest is, for example, the security constrained OPF. The effect of the $N-1$ contingencies can be naturally incorporated as a low-rank update of the matrix $M$ defining the original problem. Whenever the effect of this update admits simple bounds, the proposed framework can be easily extended to certify security in the presence of $N-1$ or higher order contingencies.
\section{Acknowledgement}
Work of KT was funded by DOE/GMLC 2.0 project: ``Emergency monitoring and controls through new technologies and analytics'', NSF CAREER award 1554171, and 1550015. HN was supported by NTU SUG and Siebel Fellowship. We thank Steven Low for his suggestions.
\bibliographystyle{IEEEtran}

\begin{thebibliography}{10}
\providecommand{\url}[1]{#1}
\csname url@samestyle\endcsname
\providecommand{\newblock}{\relax}
\providecommand{\bibinfo}[2]{#2}
\providecommand{\BIBentrySTDinterwordspacing}{\spaceskip=0pt\relax}
\providecommand{\BIBentryALTinterwordstretchfactor}{4}
\providecommand{\BIBentryALTinterwordspacing}{\spaceskip=\fontdimen2\font plus
\BIBentryALTinterwordstretchfactor\fontdimen3\font minus
  \fontdimen4\font\relax}
\providecommand{\BIBforeignlanguage}[2]{{%
\expandafter\ifx\csname l@#1\endcsname\relax
\typeout{** WARNING: IEEEtran.bst: No hyphenation pattern has been}%
\typeout{** loaded for the language `#1'. Using the pattern for}%
\typeout{** the default language instead.}%
\else
\language=\csname l@#1\endcsname
\fi
#2}}
\providecommand{\BIBdecl}{\relax}
\BIBdecl

\bibitem{cain2012history}
M.~B. Cain, R.~P. O’neill, and A.~Castillo, ``History of optimal power flow
  and formulations,'' \emph{Federal Energy Regulatory Commission}, pp. 1--36,
  2012.

\bibitem{jabr2006radial}
R.~A. Jabr, ``Radial distribution load flow using conic programming,''
  \emph{IEEE transactions on power systems}, vol.~21, no.~3, pp. 1458--1459,
  2006.

\bibitem{LavaeiLowzerogap}
J.~Lavaei and S.~H. Low, ``Zero duality gap in optimal power flow problem,''
  \emph{IEEE Transactions on Power Systems}, vol.~27, no.~1, pp. 92--107, Feb
  2012.

\bibitem{coffrin2016qc}
C.~Coffrin, H.~L. Hijazi, and P.~Van~Hentenryck, ``The qc relaxation: A
  theoretical and computational study on optimal power flow,'' \emph{IEEE
  Transactions on Power Systems}, vol.~31, no.~4, pp. 3008--3018, 2016.

\bibitem{cui2017new}
B.~Cui and X.~A. Sun, ``A new voltage stability-constrained optimal power flow
  model: Sufficient condition, socp representation, and relaxation,''
  \emph{arXiv preprint arXiv:1705.10372}, 2017.

\bibitem{Schweppe1975}
E.~Hnyilicza, S.~Lee, and F.~Schweppe, ``Steady-state security regions: The set
  theoretic approach,'' in \emph{Proc. 1975 PICA Conf.}, 1975, pp. 347--355.

\bibitem{wu1982steady}
F.~F. Wu and S.~Kumagai, ``Steady-state security regions of power systems,''
  \emph{Circuits and Systems, IEEE Transactions on}, vol.~29, no.~11, pp.
  703--711, 1982.

\bibitem{Vasin1985}
V.~Vasin, ``Regions of power system operating point existence and their
  analytic estimates,'' \emph{internal report of Energoset'proekt (in
  Russian)}, 1985.

\bibitem{Byc1989}
F.~Byc, B.~Kobets, N.~Nesterenko, and E.~Zel'manov, ``An approach for
  decentralized evaluation of regimes in power systems,'' in \emph{Fifth
  international conference on present day problems of power system automation
  and control, Gliwice}, 1989, pp. 297--301.

\bibitem{bolognani2016existence}
S.~Bolognani and S.~Zampieri, ``On the existence and linear approximation of
  the power flow solution in power distribution networks,'' \emph{IEEE
  Transactions on Power Systems}, vol.~31, no.~1, pp. 163--172, 2016.

\bibitem{EssieHungKostya}
S.~Yu, H.~D. Nguyen, and K.~S. Turitsyn, ``Simple certificate of solvability of
  power flow equations for distribution systems,'' in \emph{2015 IEEE Power
  Energy Society General Meeting}, July 2015, pp. 1--5.

\bibitem{wang2016existence}
C.~Wang, A.~Bernstein, J.-Y. Le~Boudec, and M.~Paolone, ``Existence and
  uniqueness of load-flow solutions in three-phase distribution networks,''
  \emph{IEEE Transactions on Power Systems}, 2016.

\bibitem{nguyen2017framework}
H.~D. Nguyen, K.~Dvijotham, S.~Yu, and K.~Turitsyn, ``A framework for robust
  steady-state voltage stability of distribution systems,'' \emph{arXiv
  preprint arXiv:1705.05774}, 2017.

\bibitem{simpson2016voltage}
J.~W. Simpson-Porco, F.~D{\"o}rfler, and F.~Bullo, ``Voltage collapse in
  complex power grids,'' \emph{Nature communications}, vol.~7, 2016.

\bibitem{simpson2017part1}
J.~W. Simpson-Porco, ``A theory of solvability for lossless power flow
  equations--part i: Fixed-point power flow,'' \emph{IEEE Transactions on
  Control of Network Systems}, 2017.

\bibitem{simpson2017part2}
J.~W. Simpson-Porco, ``A theory of solvability for lossless power flow
  equations--part ii: Conditions for radial networks,'' \emph{IEEE Transactions
  on Control of Network Systems}, vol.~PP, no.~99, pp. 1--1, 2017.

\bibitem{DjTuritsyn2015Feasibility}
K.~{Dvijotham} and K.~{Turitsyn}, ``{Construction of power flow feasibility
  sets},'' \emph{ArXiv e-prints}, Jun. 2015.

\bibitem{Dvijotham2018LCSS}
K.~Dvijotham, H.~Nguyen, and K.~Turitsyn, ``Solvability regions of affinely
  parameterized quadratic equations,'' \emph{IEEE Control Systems Letters},
  vol.~2, no.~1, pp. 25--30, Jan 2018.

\bibitem{ChiangOPF}
H.~D. Chiang and C.~Y. Jiang, ``Feasible region of optimal power flow:
  Characterization and applications,'' \emph{IEEE Transactions on Power
  Systems}, vol.~PP, no.~99, pp. 1--1, 2017.

\bibitem{spanier1994algebraic}
E.~H. Spanier, \emph{Algebraic topology}.\hskip 1em plus 0.5em minus
  0.4em\relax Springer Science \& Business Media, 1994, vol.~55, no.~1.

\bibitem{6949701}
W.~Wei, F.~Liu, and S.~Mei, ``Dispatchable region of the variable wind
  generation,'' \emph{IEEE Transactions on Power Systems}, vol.~30, no.~5, pp.
  2755--2765, Sept 2015.

\bibitem{7079527}
W.~Wei, F.~Liu, and S.~Mei, ``Real-time dispatchability of bulk power systems
  with volatile renewable generations,'' \emph{IEEE Transactions on Sustainable
  Energy}, vol.~6, no.~3, pp. 738--747, July 2015.

\bibitem{6856230}
J.~Zhao, T.~Zheng, and E.~Litvinov, ``Variable resource dispatch through
  do-not-exceed limit,'' \emph{IEEE Transactions on Power Systems}, vol.~30,
  no.~2, pp. 820--828, March 2015.

\bibitem{7024953}
J.~Zhao, T.~Zheng, and E.~Litvinov, ``A unified framework for defining and
  measuring flexibility in power system,'' \emph{IEEE Transactions on Power
  Systems}, vol.~31, no.~1, pp. 339--347, Jan 2016.

\bibitem{7453153}
F.~Qiu, Z.~Li, and J.~Wang, ``A data-driven approach to improve wind
  dispatchability,'' \emph{IEEE Transactions on Power Systems}, vol.~32, no.~1,
  pp. 421--429, Jan 2017.

\bibitem{andrews1992special}
L.~Andrews, \emph{Special Functions of Mathematics for Engineers}, ser. Oxford
  science publications.\hskip 1em plus 0.5em minus 0.4em\relax SPIE Optical
  Engineering Press, 1992.

\bibitem{ajjarapu1992continuation}
V.~Ajjarapu and C.~Christy, ``The continuation power flow: a tool for steady
  state voltage stability analysis,'' \emph{IEEE transactions on Power
  Systems}, vol.~7, no.~1, pp. 416--423, 1992.

\end{thebibliography}

\appendices
\section{Optimal construction validation} \label{app:validatecons}
To validate the construction of optimal inner approximation \eqref{eq:optimalextension}, which is the main result introduced in section \ref{sec:affinput}, we will show that the imposed constraints guarantee the existence of a feasible solution. In particular, as we show below, such constraints are sufficient conditions for solvability and feasibility.

\subsection{Solvability condition}
The goal of this section is to verify that the condition \eqref{eq:constsol} below, which is also the first set of constraints in \eqref{eq:optimalextension}, is sufficient to ensure the self-mapping required by the Brouwer theorem and thus guarantees the solvability of the power flow equations \eqref{eq:abstract}.
\begin{equation} \label{eq:constsol}
\limy^\pm \geq B^+ \limu^\pm + B^- \limu^\mp +C^+ \delta_2 f^{\pm}(\limy) + C^{-}\delta_2 f^{\mp}(\limy).
\end{equation}
Specifically, we rely on Corollary \ref{cor:selfmaptheo} and later Corollary \ref{cor:ubox} to show that the above constraint verifies self-mapping condition, thus defining a solvability subset $\mathcal{U}(\limu)$. The following Lemma \ref{le:tau} and Theorem \ref{th:main} are necessary to derive the central result of this section, which is presented in Corollary \ref{cor:ubox}.

\begin{lemma}\label{le:tau}
Given the nonlinear system described by the equation \eqref{eq:fixedpoint}, the admissibility and nonlinear bound polytope families $\mathcal{A}(\limy)$ and $\mathcal{N}(\limy)$, if $\xdiff \in \Acal(\limy)$, then the nonlinear correction term $\phi(\xdiff)$ is contained in the polytope $\Acal(\tau(\limy))$, i.e. $\phi(\xdiff) \in \Acal(\tau(\limy))$ with $\tau(\limy) = [-\tau^-(\limy), \tau^+(\limy)] \in \mathbb{R}_+^{k\times 2}$ given by
\begin{equation}\label{eq:tau}
 \tau^\pm(\limy) = C^+ \delta_2 f^{\pm}(\limy) + C^{-}\delta_2 f^{\mp}(\limy)  
\end{equation}
where $ C^+ - C^- = -AJ_\star^{-1}M$ and $C^{\pm}_{ij} \geq 0$.
\begin{proof}
Consider the nonlinear correction term below
\begin{equation}
A\phi(\udiff) = -A J_\star^{-1} M \delta_2f(\xdiff)
              = C \delta_2f(\xdiff)
\end{equation}
Consider the decomposition of the matrix $C$ in $C = C^+ - C^-$ with $C^\pm_{ij} \geq 0$. The component-wise upper bounds of $C \delta_2 f$ can be then represented as 
\begin{equation}
    C \delta_2 f \leq C^+ \max \delta_2 f - C^- \min \delta_2 f.
\end{equation}
Given the bounds $-\delta_2f^-(\limy)- \leq \delta_2f(\xdiff) \leq \delta_2f^+(\limy)$ this transforms into 
\begin{equation}
    C \delta_2 f \leq C^+ \delta_2 f^+ + C^-  \delta_2 f^-.
\end{equation}
Similarly, one derives the lower bound on $C\delta_2 f$. Both of the bounds can be combined in a single expression of the form \eqref{eq:tau}.
\end{proof}
\end{lemma}

\begin{remark} The map $\tau(\limy): \mathbb{R}_+^{k\times 2}\to\mathbb{R}_+^{k\times 2}$ is monotone in $\limy$ as increasing the region of variable variations defined by $\limy$ only increases the bounds on the nonlinear residual terms. 
\end{remark}

\begin{theorem} \label{th:main}
Given all the definitions above and sufficiently tight bounds on the nonlinearity, there exists a matrix $\limy$ such that the matrix $\tau(\limy)$, which is defined in the Lemma \ref{le:tau}, satisfies  $\tau^\pm (\limy) \leq \limy^\pm$. Then, for any $\udiff$ such that $J_\star^{-1}R\udiff \in \mathcal{A}(\limy - \tau(\limy))$, there exists at least one admissible solution $\xdiff \in \mathcal{A}(\limy)$ of the equation \eqref{eq:fixedpoint}.
\begin{proof}
The ``sufficiently tight'' conditions refers to Remark 1 in the main text and requires the nonlinear bounds to be superlinear in the vicinity of the operating point.

The main statement of the theorem follows from the Brouwer fixed point theorem. Indeed, whenever $\tau^\pm(\limy) \leq \limy^\pm$, the right hand side of \eqref{eq:fixedpoint} is contained in the polytope $\mathcal{A}(\limy)$, therefore the map $F(\xdiff) = J_\star^{-1}R\udiff + \phi(\xdiff)$ maps the compact convex set $\mathcal{A}(\limy)$ onto itself. Thus,  according to Brouwer's theorem, there exists a fixed point $\xdiff = F(\xdiff)$ inside $\mathcal{A}(\limy)$.
\end{proof}

\end{theorem}


\begin{corollary}\label{cor:ubox}
The admissible solution $\xdiff^\star \in \mathcal{A}(\limy)$ is guaranteed to exists for every element $\udiff$ inside a box-shaped uncertainty set $\mathcal{U}(\limu)$ whenever the condition
\begin{equation} \label{eq:selfmap}
    \limy^\pm \geq \sigma^\pm(\limu) + \tau^\pm(\limy)
\end{equation}
\begin{equation}\label{eq:sigma}
\sigma^\pm(\limu) = B^+ \limu^\pm + B^- \limu^\mp \end{equation}
is satisfied with  $B^+ - B^- = A J_\star^{-1} R$ and $B^\pm_{ij} \geq 0$.
\begin{proof}
The logic here is the same as in Lemma \ref{le:tau}. Whenever $\udiff \in \mathcal{U}(\limu)$, one has $-\sigma^-(\udiff) \leq A J_\star^{-1} R \udiff \leq \sigma^+(\udiff)$, hence the inputs satisfy the conditions of Theorem \ref{th:main}.
\end{proof}
\end{corollary}
Substituting \eqref{eq:sigma} and \eqref{eq:tau} into the self-mapping condition \eqref{eq:selfmap}, yields the sufficient condition for solvability \eqref{eq:constsol}.

\subsection{Feasibility conditions}
As the self-mapping condition \eqref{eq:selfmap} verifies the existence of a solution, we complete validating the central optimization problem \eqref{eq:optimalextension} by proving the feasibility. In particular, the voltage level and angle separation limits are enforced by constructing a proper matrix $A$ and limiting the corresponding components of $\limy$. These bounds are also naturally used to bound the terms $\delta\rho_e$ and $\delta\theta_e$ which are the arguments of the primitive nonlinear functions.


Next, we consider the remaining nonlinear feasibility constraints $h(x) \leq 0$, which is introduced in \eqref{eq:feasconst}. The following Lemma facilitates the incorporation of these feasibility conditions and completes our construction \eqref{eq:optimalextension}.

\begin{lemma} (Feasibility conditions) \label{lem:feasconst}
The nonlinear feasibility constraints $h(x) \leq 0$ will be satisfied if the following condition holds.
\begin{equation} \label{eq:feasconstup}
D^+ \limu^+ + D^- \limu^- + E^+\delta_2 f^+(\limy) + E^-\delta_2 f^-(\limy) + h_\star \leq 0
 \end{equation}
where $D = T L J_\star^{-1}R$, $E = T - T L J_\star^{-1}M$, and $L = \frac{\partial f}{\partial x}|_{x_\star}$.
\end{lemma}
\begin{proof}
Central to the proof of this result is the following observation. Using the definition of the residual operator $\delta_2 f$, the nonlinear constraints can be expressed as follows.
\begin{align}
h(x) &= h_\star + T L \tilde{x} + T \delta_2 f \nonumber\\
     &= h_\star + T L (J_\star^{-1} R \udiff + \phi(\tilde{x})) + T \delta_2 f \nonumber\\
     &= h_\star + T L J_\star^{-1}R\delta u + (T - T L J_\star^{-1}M) \delta_2 f.
\end{align}
The second equality can be derived with the help of the fixed point equation \eqref{eq:fixedpoint}. Next, using the same bounding approaches used in the previous lemmas, we arrive at the bound \eqref{eq:feasconstup} where the left hand side is an upper bound for $h(x)$.Thus, the condition \eqref{eq:feasconstup} is sufficient for $h(x) \leq 0$ to hold at the fixed point.
\end{proof}

\section{Nonlinear bounding and LP relaxation} \subsection{Nonlinear bounding}\label{app:nlbounds}
The bounds on the individual nonlinear terms of $f(x)$ can be expressed as:
\begin{subequations}
\begin{align}
    \delta^+_{1,2}\{\cos\delta\theta_e\} &= 0 \\
    \delta^-_{1,2}\{\cos\delta\theta_e\} &= \max\{1-\cos\delta\theta^+_e, 1-\cos\delta\theta^-_e\} \\
    \delta^\pm\{\sin\delta\theta_e\} &= \sin\delta\theta^\pm_e \\
    \delta^\pm_{2}\{\sin\delta\theta_e\} &= \delta\theta^\mp_e - \sin\delta\theta^\mp_e \\
    \delta^-_{1,2}\{\cosh\delta\rho_e\} &= 0 \\
    \delta^+_{1,2}\{\cosh\delta\rho_e\} &= \max\{\cosh\delta\rho^+_e, \cosh\delta\rho^-_e \} - 1 \\
    \delta^\pm\{\sinh\delta\rho_e\} &= \sinh\delta\rho^\pm_e \\
    \delta^\pm_{2}\{\sinh\delta\rho_e\} &=  \sinh\delta\rho^\pm_e - \delta\rho^\pm_e
\end{align}
\end{subequations}
These bounds generally follow from the monotonicity/convexity of each of the univariate functions involved. The following corollary and lemma are used to construct the bounds on the multi-variate functions appearing in our power flow representation.
\begin{corollary} \label{cor:D2}
For the element-wise product of two vector functions $f(x)\odot h(x)$ one has 
\begin{align}\label{eq:d2f}
 \delta_2\{f\odot g\}= \delta f \odot \delta g + \delta_2 f\odot g^\star + 
 f^\star\odot \delta_2 g
\end{align}
\end{corollary}
\begin{proof}
This follows directly from the definition of the first $\delta f, \delta g$ and second $\delta_2 f, \delta_2g$ order residuals.
\end{proof}
\begin{lemma} \label{lem:boundprod}
Bounds on the products can be characterized in the following way:
\begin{align} 
    \delta_2^\pm\{f\odot g\} 
    &= \max\{\delta f^+\odot \delta g^\pm, \delta f^- \odot \delta g^\mp \} \nonumber\\
    &+ f^\star \odot \delta_2 g^\pm +  \delta_2 f^\pm  \odot g^\star
\end{align}
\end{lemma}
\begin{proof}
This follows directly from \eqref{eq:d2f} and sign-definiteness of the residual bounds, i.e. $\delta f^\pm, \delta g^\pm, \delta_2 f^\pm, \delta_2 g^\pm \geq 0$.
\end{proof}

\subsection{LP relaxation of the original problem} \label{app:ext}
As discussed in the main text, the Linear Programming relaxation can be used to find simple certificates for initialization of the nonlinear search. Here we derive the linear bounds for the nonlinear function $\phi(\xdiff)$ which will form the LP version of problem \eqref{eq:optimalextension}. These linear bounds can be derived whenever the vector $\limy$ is itself bounded. Assume that the following bounds hold for every polytope:
\begin{subequations}
\begin{align}
   -\delta\theta_e^U \leq  -\delta\theta_e^- \leq \delta \theta_e \leq \delta\theta_e^+ \leq \delta\theta_e^U, \\
 -\delta\rho_e^U \leq  -\delta\rho_e^- \leq \delta\rho_e \leq \delta\rho_e^+ \leq \delta\rho_e^U.
\end{align}
\end{subequations}
Also, define $\delta\theta_e^m = \max\{\delta\theta_e^\pm\}$ and $\delta\rho_e^m = \max\{\delta\rho_e^\pm\}$. Then the following bounds hold for individual nonlinear terms:
\begin{subequations}
\begin{align}
    \delta^+_{1,2}\{\cos\delta\theta_e\} &= 0 \\
    \delta^-_{1,2}\{\cos\delta\theta_e\} &= (1-\cos(\delta\theta_e^U))\frac{\delta\theta_e^m}{\delta\theta_e^U} \\
    \delta^\pm\{\sin\delta\theta_e\} &= \delta\theta^\pm_e \\
    \delta^\pm_{2}\{\sin\delta\theta_e\} &= (\delta\theta_e^U - \sin(\delta\theta_e^U))\frac{\delta\theta^\mp_e}{\delta\theta_e^U} \\
    \delta^-_{1,2}\{\cosh\delta\rho_e\} &= 0 \\
    \delta^+_{1,2}\{\cosh\delta\rho_e\} &= (\cosh(\delta\rho_e^U) - 1)\frac{\delta\rho^m}{\delta\rho_e^U} \\
    \delta^\pm\{\sinh\delta\rho_e\} &= \sinh(\delta\rho_e^U)\frac{\delta\rho^\pm_e}{\delta\rho_e^U} \\
    \delta^\pm_{2}\{\sinh\delta\rho_e\} &= (\sinh(\delta\rho_e^U)-\delta\rho_e^U)\frac{\delta\rho^\pm_e}{\delta\rho_e^U}
\end{align}
\end{subequations}
where we also assumed $ \delta\theta_e^U \leq \pi/2$ and $\delta\rho_e^U \leq 1$. 

To construct the bounds for the product terms, we rely on the standard McCormick envelopes: 
\begin{subequations}
\begin{align}
xy &\leq x^U y + xy^L - x^U y^L, \\
xy &\leq x y^U + x^Ly - x^L y^U   
\end{align}
\end{subequations}
where the $^{L}$ and $^{U}$ superscripts stand for the upper and lower bounds of the corresponding terms.

\begin{IEEEbiography}[{\includegraphics[width=1in,height=1.25in,clip,keepaspectratio]{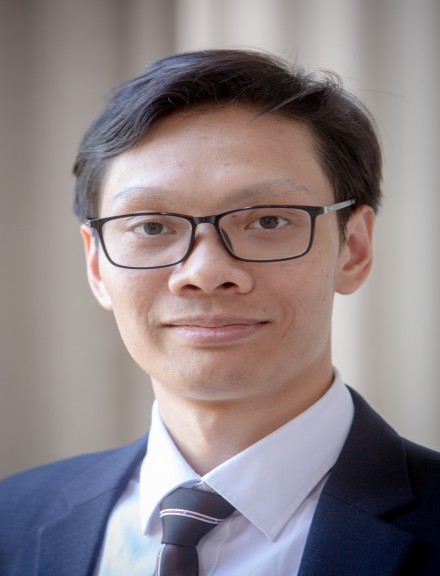}}]{Hung D. Nguyen} (S`12) received the B.E. degree in electrical engineering from HUT, Vietnam, in 2009, the M.S. degree in electrical engineering from Seoul National University, Korea, in 2013, and the Ph.D. degree in Electric Power Engineering at Massachusetts Institute of Technology (MIT). Currently, he is an Assistant Professor in Electrical and Electronic Engineering at NTU, Singapore. His current research interests include power system operation and control; the nonlinearity, dynamics and stability of large scale power systems; DSA/EMS and smart grids.
\end{IEEEbiography}
\begin{IEEEbiography}[{\includegraphics[width=1in,height=1.25in,clip,keepaspectratio]{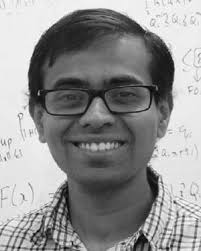}}]{Krishnamurthy Dvijotham} received the bachelor’s degree from the Indian Institute of Technology Bombay, Mumbai, India, in 2008, and the Ph.D. degree in computer science and engineering from the University of Washington, Seattle, Washington, USA, in 2014. He was a Researcher with the Optimization and Control Group, Pacific Northwest National Laboratory and a Research Assistant Professor with Washington State University, Pullman, WA, USA. He currently is a research scientist in Google Deepmind. His research interests are in developing efficient algorithms for automated decision making in complex physical, cyber-physical and social systems. 
\end{IEEEbiography}
\begin{IEEEbiography}[{\includegraphics[width=1in,height=1.25in,clip,keepaspectratio]{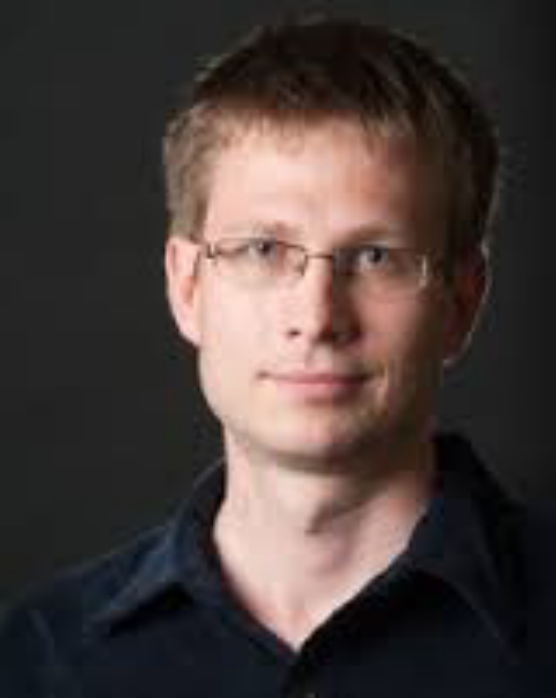}}]{Konstantin Turitsyn} (M`09) received the M.Sc. degree in physics from Moscow Institute of Physics and Technology and the Ph.D. degree in physics from Landau Institute for Theoretical Physics, Moscow, in 2007.  Currently, he is an Associate Professor in the Mechanical Engineering Department of Massachusetts Institute of Technology (MIT), Cambridge. Before joining MIT, he held the position of Oppenheimer fellow at Los Alamos National Laboratory, and Kadanoff–Rice Postdoctoral Scholar at University of Chicago. His research interests encompass a broad range of problems involving nonlinear and stochastic dynamics of complex systems. Specific interests in energy related fields include stability and security assessment, integration of distributed and renewable generation.
\end{IEEEbiography}

\end{document}